\documentclass{amsart}

\newcommand{\C}{\mathbb{C}}
\newcommand{\D}{\mathbb{D}}

\newcommand{\N}{\mathbb{N}}

\newcommand{\dinf}{d_{\operatorname{inf}}}
\newcommand{\Area}{\operatorname{Area}}

\newtheorem{theorem}{Theorem}[section]
\newtheorem{lemma}[theorem]{Lemma}

\theoremstyle{definition}
\newtheorem{definition}[theorem]{Definition}

\theoremstyle{theorem}

\theoremstyle{theorem}

\theoremstyle{theorem}

\theoremstyle{theorem}

\theoremstyle{definition}

\theoremstyle{theorem}

\numberwithin{equation}{section}

\begin{document}
\title{Carath\'eodory's Theorem and moduli of local connectivity}
\author{Timothy H. McNicholl}
\address{Department of Mathematics\\
Iowa State University\\
Ames, Iowa 50011}
\email{mcnichol@iastate.edu}

\begin{abstract}
We give a quantitative proof of the Carath\'eodory Theorem by means of the concept of a modulus of local connectivity and the extremal distance of the separating curves of an annulus.
\end{abstract}
\subjclass{30}
\maketitle

\section{Introduction}

The goal of this paper is to give a new proof of the Carath\'eodory Theorem which states that if $D$ is a Jordan domain, and if $\phi$ is a conformal map of $D$ onto the unit disk, then $\phi$ extends to a homeomorphism of $\overline{D}$ with the closed unit disk (see e.g. \cite{Golusin.1969}, \cite{Greene.Krantz.2002}, and \cite{Palka.1991}).  
This proof has a feature which appears to be new in that for each $\zeta \in \partial D$ it explicitly constructs a $\delta$ for each $\epsilon$ when proving the existence of $\lim_{z \rightarrow \zeta} \phi(z)$.  Furthermore, a closed form expression for $\delta$ in terms of $\epsilon$ and $\zeta$ is obtained.  Such expressions are potentially useful when estimating error in numerical computations.
This is accomplished by means of a \emph{modulus of local connectivity} for the boundary of $D$.  Roughly speaking, this is a function that predicts how close two boundary points must be in order to connect them with a small arc that is included in the boundary.  
As in \cite{Palka.1991}, the proof uses the extremal distance of the separating curves of an annulus to bound $|\phi(z) - \phi(\zeta)|$. 

The paper is organized as follows.  Section \ref{sec:BACKGROUND} covers background material.  Section \ref{sec:OUTLINE} states the main ideas of the proof. Sections \ref{sec:DISK-ARC} and \ref{sec:POLAR} deal with topological preliminaries.  Our estimates are proven in Section \ref{sec:ESTIMATES} and Section \ref{sec:CT} completes the proof.  

\section{Background}\label{sec:BACKGROUND}

Let $\N$ denote the set of non-negative integers.

When $\mathcal{A}$ is an annulus with inner radius $r$ and outer radius $R$, let 
\[
\lambda(\mathcal{A}) = \frac{2\pi}{\log(R/r)}.
\]  
$\lambda(\mathcal{A})$ is the extremal length of the family of separating curves of $\mathcal{A}$; see e.g. \cite{Garnett.Marshall.2005}.  Note that $\lambda(\mathcal{A})$ decreases as the annulus $\mathcal{A}$ gets thicker (i.e. as the ratio $R/r$ increases) and increases as $\mathcal{A}$ gets thinner (i.e. as the ratio $R/r$ decreases).

When $X$, $Y$, and $Z$ are subsets of the plane, we say that $X$ \emph{separates} $Y$ from $Z$ if $Y$ and $Z$ are included in distinct connected components of $\C - X$.  In the case where $Y = \{p\}$, we say that $X$ separates $p$ from $Z$.  In the case where $Y = \{p\}$ and $Z = \{q\}$ we say that $X$ separates $p$ from $q$. 

A topological space is \emph{locally} connected if it has a basis of open connected sets.  By the Hahn-Mazurkiewicz Theorem, every curve is locally connected; see e.g. Section 3-5 of \cite{Hocking.Young.1988}.  Suppose $X$ is a compact and connected metric space.  Then, $X$ is locally connected if and only if it is \emph{uniformly locally arcwise connected}.  This means that for every $\epsilon > 0$, there is a $\delta > 0$ so that whenever $p,q \in X$ and $0 < d(p,q) < \delta$, $X$ includes an arc from $p$ to $q$ whose diameter is smaller than $\epsilon$ (although its length may be infinite); again, see Section 3-5 of \cite{Hocking.Young.1988}.  Accordingly, we define a \emph{modulus of local connectivity} for a metric space $X$ to be a function $f : \N \rightarrow \N$  so that whenever $p,q \in X$ and $0 < d(p,q) \leq 2^{-f(k)}$, $X$ includes an arc from $p$ to $q$ whose diameter is smaller than $2^{-k}$.  Thus, a metric space is uniformly locally arcwise connected if and only if it has a modulus of local connectivity, and a metric space that is compact and connected is locally connected if and only if it has a modulus of local connectivity.  Note that if $f$ is a modulus of local connectivity, then $\lim_{k \rightarrow \infty} f(k) = \infty$.   In addition, if a metric space has a modulus of local connectivity, then it has a modulus of local connectivity that is increasing.    

Moduli of local connectivity originated in the adaptation of local connectivity properties to the setting of theoretical computer science in \cite{Couch.Daniel.McNicholl.2012} and \cite{Daniel.McNicholl.2012}.   Computational connections between moduli of local connectivity and boundary extensions of conformal maps are made in \cite{McNicholl.2014}.  Here, we show that this notion may be useful in more traditional mathematical settings.

\section{Outline of the proof}\label{sec:OUTLINE}

We first observe the following which is proven in Section \ref{sec:DISK-ARC}.

\begin{theorem}\label{thm:BOUNDARY.COMP}
If $\zeta_0$ is a boundary point of a simply connected Jordan domain $D$, then for every $r > 0$, $\zeta_0$ is a boundary point of exactly one connected component of $D_r(\zeta_0) \cap D$.
\end{theorem}

Suppose $\zeta_0$ is a boundary point of a simply connected Jordan domain $D$.   In light of Theorem \ref{thm:BOUNDARY.COMP}, when $r >0$ we let $C(D; \zeta_0, r)$ denote the connected component of $D_r(\zeta_0) \cap D$ whose boundary contains $\zeta_0$.  
Suppose $\phi$ is a conformal map of $D$ onto the unit disk.  The fundamental strategy of the proof is to bound the diameter of $\phi[C(D; \zeta_0, r)]$.  To do so, we first construct an upper bound on the diameter of $\phi[C]$ where $C$ is a connected component of $D_r(\zeta) \cap D$ for some point $\zeta$ in the complement of $D$.  Namely, in Section \ref{sec:ESTIMATES} we prove the following.

\begin{theorem}\label{thm:MODULUS.ESTIMATE}
Let $\phi$ be a conformal map of a 
domain $D$ onto the unit disk.  
Suppose $\mathcal{A}$ is an annulus so that $\overline{\mathcal{A}}$ separates 
its center from $\phi^{-1}[\overline{D_r(0)}]$ where $r \geq \sqrt{\pi\lambda(\mathcal{A})}$.      
Let 
$C$ be a connected component of the points of $D$ that are inside the inner circle of $\mathcal{A}$.  Suppose $l = 1 - \sqrt{r^2 - \pi \lambda(\mathcal{A})}$.
Then, the diameter of $\phi[C]$ is at most $\displaystyle{\sqrt{l^2 + 4 \pi \lambda(\mathcal{A}) }}$.
\end{theorem}

Note that Theorem \ref{thm:MODULUS.ESTIMATE} applies to non-Jordan domains.  

With Theorem \ref{thm:MODULUS.ESTIMATE} in hand, some basic calculations, which we perform in Section \ref{sec:ESTIMATES}, lead us to the following.

\begin{theorem}\label{thm:DIAM.EST}
Suppose $\phi$ is a conformal map of a Jordan domain $D$ onto the unit disk.  Let $\zeta_0$ be a boundary point of $D$, and let $\epsilon > 0$.  Then, the diameter of $\phi[C(D; \zeta_0, r_0)]$ is smaller than $\epsilon$ whenever $r_0 $ is a positive number that is smaller than  
\begin{equation}
 \sup_{0 < l < \epsilon} \left( 
\exp\left( \frac{8\pi^2}{l^2 - \epsilon^2} \right) \min\left\{|\zeta_0 - \phi^{-1}(w)|\ :\ |w| \leq \sqrt{(1 - l)^2 + \frac{\epsilon^2 - l^2}{4}}\right\} 
\right).\label{eqn:DIAM.EST}
\end{equation}
\end{theorem}

When $0 < \epsilon < 1$ and $l = \frac{\epsilon}{2}$, 
\[
\frac{7}{16} < (1 - l)^2 + \frac{\epsilon^2 - l^2}{4} < 1.
\]
Thus, (\ref{eqn:DIAM.EST}) is positive when $0 < \epsilon < 1$.  In other words, for all sufficiently small $\epsilon > 0$, there \emph{is} a positive number $r_0$ that is smaller than (\ref{eqn:DIAM.EST}).

So, suppose $\phi$ is a conformal map of a Jordan domain $D$ onto the unit disk.  We use Theorem \ref{thm:DIAM.EST} to form an extension of $\phi$ to $\overline{D}$ as follows.  Let $\zeta_0$ be a boundary point of $D$.  Note that $C(D;\zeta_0, r') \subseteq C(D; \zeta_0, r)$ when $0 < r' < r$.  It follows from Theorem \ref{thm:DIAM.EST} that there is exactly one point in 
\[
\bigcap_{r > 0} \overline{\phi[C(D; \zeta_0, r)]}.
\]
We define this point to be $\phi(\zeta_0)$.  

Our next goal is to show that this extension of $\phi$ is continuous.  That is, $\lim_{z \rightarrow \zeta} \phi(z) = \phi(\zeta)$ whenever $\zeta$ is a boundary point of $D$.  
This is accomplished by showing that $z \in C(D; \zeta, r)$ whenever $z \in D$ is sufficiently close to $\zeta$.  This is where we use moduli of local connectivity.  Namely, in Section \ref{sec:DISK-ARC} we prove the following.

\begin{theorem}\label{thm:MLC}
Suppose $g$ is a modulus of local connectivity for a Jordan curve $\sigma$.  Suppose $D$ is an open disk whose boundary separates two points of $\sigma$.  Suppose $z_0$ and $\zeta_0$ are points so that $\zeta_0 \in \sigma \cap D$, $z_0 \in D - \sigma$, and $|z_0 - \zeta_0| < 2^{-g(k)}$ where $2^{-k} + 2^{-g(k)} \leq \max\{d(\zeta_0, \partial D), d(z_0, \partial D)\}$.  Then, $\zeta_0$ is a boundary point of the connected component of $z_0$ in $D - \sigma$.
\end{theorem}

Theorem \ref{thm:MLC} was previously proven by means of the Carath\'eodory Theorem in \cite{McNicholl.2013.MLQ}.  We give another proof here with a few extra topological steps so as to avoid circular reasoning.  

We then obtain the following form of the Carath\'eodory Theorem from Theorems \ref{thm:DIAM.EST} and Theorem \ref{thm:MLC}.

\begin{theorem}\label{thm:CT}
Suppose $\phi$ is a conformal map of a Jordan domain $D$ onto the unit disk.  Let $\zeta_0$ be a boundary point of $D$.  Then, $\lim_{z \rightarrow \zeta_0} \phi(z) = \phi(\zeta_0)$.  Furthermore, if $g$ is a modulus of local connectivity for the boundary of $D$, then for each $\epsilon > 0$, $|\phi(z_0) - \phi(\zeta_0)| < \epsilon$ whenever $z_0$ is a point in $D$ so that $|z_0 - \zeta_0| < 2^{-g(k)}$ and $k$ is a non-negative integer so that $2^{-k} + 2^{-g(k)}$ is smaller than (\ref{eqn:DIAM.EST}).  Finally, the extension of $\phi$ to $\overline{D}$ is a homeomorphism of $\overline{D}$ with the closed unit disk.
\end{theorem}

The proof of Theorem \ref{thm:CT} is given in Section \ref{sec:CT}.

Suppose $\phi$, $D$, $g$, $\zeta_0$ are as in Theorem \ref{thm:CT}.  Without loss of generality suppose $g$ is increasing.  Thus $2^{-k} + 2^{-g(k)} \leq 2^{-k + 1}$.  Let $0 < \epsilon < 1$.  We define a positive number $\delta(\zeta_0, \epsilon)$ so that $|\phi(z) - \phi(\zeta_0)| < \epsilon$ when $|z - z_0| < \delta(\zeta_0, \epsilon)$.  Let:  
\begin{eqnarray*}
k(\zeta_0, \epsilon) & = & 2 - \left\lfloor \sup_{0 < l < \epsilon} \left( \frac{8\pi^2}{l^2 - \epsilon^2} + \right.\right.\\
& & \left.\left. \min\left\{ \log|\zeta_0 - \phi^{-1}(w)|\ :\ |w| \leq \sqrt{(1 - l)^2 + \frac{\epsilon^2 - l^2}{4}}\right\} \right)\right\rfloor\\
\delta(\zeta_0, \epsilon) & = & 2^{-k(\zeta_0, \epsilon)} + 2^{-g(k(\zeta_0, \epsilon))}
\end{eqnarray*}
(Here, $\lfloor x \rfloor$ denotes the largest integer that is not larger than $x$.)  Thus, by Theorem \ref{thm:CT}, $|\phi(z) - \phi(\zeta_0)| < \epsilon$ whenever $z \in D$ and $|z - \zeta_0| < \delta(\zeta_0, \epsilon)$.

\section{Proofs of Theorems \ref{thm:BOUNDARY.COMP} and \ref{thm:MLC}}\label{sec:DISK-ARC}

Theorem \ref{thm:MLC} is used to prove Theorem \ref{thm:BOUNDARY.COMP}.  The proof of Theorem \ref{thm:MLC} is based on the following lemma and theorem.  

\begin{lemma}\label{lm:SIDES}
Let $D$ be a Jordan domain.  Let $\alpha$ be a crosscut of $D$, and let $\gamma_1$, $\gamma_2$ be the subarcs of the boundary of $D$ that join the endpoints of $\alpha$.  Then, the interior of $\gamma_1 \cup \alpha$ is one side of $\alpha$, and the interior of $\gamma_2 \cup \alpha$ is the other side of $\alpha$.
\end{lemma}

\begin{proof}
Let $U_j$ denote the interior of $\alpha \cup \gamma_j$.  Choose a point $p$ in $\alpha \cap D$. There is a positive number $\delta$ so that $D_\delta(p) \subseteq D$.  Since $p$ is a boundary point of $U_j$, $U_j \cap D_\delta(p)$ is non-empty.   So, let $q_j \in U_j \cap D_\delta(p)$, and let $D_j$ be the side of $\alpha$ that contains $q_j$.

We show that $U_j = D_j$.  $U_j$ is a connected subset of $D - \alpha$ that contains a point of $D_j$ (namely $q_j$).  So, $U_j \subseteq D_j$.  On the other hand, $D_j$ is a connected subset of $\C - (\gamma_j \cup \alpha)$ that contains a point of $U_j$.  So, $D_j \subseteq U_j$.  

$D_1 \neq D_2$ since $\partial D_1 \neq \partial D_2$.   Thus, $U_1$ and $U_2$ are the two sides of $\alpha$. 
\end{proof}

\begin{theorem}\label{thm:COMP.JORDAN}
Let $D$ be an open disk, and let $\sigma$ be a Jordan curve.  Suppose the boundary of $D$ separates two points of $\sigma$.  Let $C$ be a connected component of $D - \sigma$.
Then, $C$ is the interior of a Jordan curve.  Furthermore, if $p$ is a boundary point of $C$ that also lies in $D$, then $p$ lies on $\sigma$ and the boundary of $C$ includes the connected component of $p$ in $D \cap \sigma$.
\end{theorem}

\begin{proof}
Since $C \neq D$, the boundary of $C$ contains a point of $\sigma$; let $p$ denote such a point.

Since the boundary of $D$ separates two points of $\sigma$, if $G$ is a connected component of $D \cap \sigma$, then $\overline{G}$ is a crosscut of $D$.  

Let $E$ denote the connected component of $p$ in $\sigma \cap D$.  Since $C$ is a connected subset of $D - E$, there is a side of $E$ that includes $C$; let $E^-$ denote this side, and let $E^+$ denote the other side.  By Lemma \ref{lm:SIDES}, each of these sides is a Jordan domain.  Again, since the boundary of $D$ separates two points of $\sigma$, if $G$ is a connected component of $\sigma \cap E^-$, then $\overline{G}$ is a crosscut of $E^-$.  

We aim to show that the boundary of $C$ is a Jordan curve which includes $E$.  To this end, we construct an arc $F$ so that $E \cup F$ is a Jordan curve whose interior is $C$.  $F$ will be a union of subarcs of $\sigma$ and connected subsets of the boundary of $D$.  To define these subarcs of $\sigma$, we define a partial ordering of the connected components of $\sigma \cap E^-$.  Namely, when $G_1, G_2$ are connected components of $\sigma \cap E^-$, write $G_1 \prec G_2$ if $G_2$ is between $G_1$ and $E$; that is if $E$ and $G_1$ lie in opposite sides of $\overline{G_2}$.

Since $\sigma$ is locally connected, it follows that there is no increasing chain $G_1 \prec G_2 \prec G_3 \prec \ldots$.  It then follows that if $G_1$ is a connected component of $\sigma \cap E^-$, then there is a $\preceq$-maximal component of $\sigma \cap E^-$, $G$, so that $G_1 \preceq G$.

We now define $F$.  
Let $F' = \partial E^- \cap \partial D$.  Thus, $E \cup F' = \partial E^-$.  Let $\mathcal{M}$ denote the set of all $\preceq$-maximal components of $\sigma \cap E^-$.  For each $G \in \mathcal{M}$, let $\lambda_G$ be the subarc of $F'$ that joins the endpoints of $\overline{G}$.  Let $F$ be formed by removing each $\lambda_G$ from $F'$ and replacing it with $\overline{G}$.

Thus, $F$ is an arc that joins the endpoints of $E$ and that contains no other points of $E$.  Let $J = E \cup F$.  Then, $J$ is a Jordan curve.  We show that $C$ is the interior of $J$.  Note that since $J \subseteq \overline{E^-}$, $E^-$ includes the interior of $J$.

When $G \in \mathcal{M}$, let $G^+$ be the side of $\overline{G}$ that includes $E$ (when $\overline{G}$ is viewed as a crosscut of $D$ rather than $E^-$), and let $G^-$ denote the other side.  The rest of the proof revolves around the following four claims.
\begin{enumerate}
	\item For each $G \in \mathcal{M}$, the exterior of $J$ includes $G^-$.\label{cl1}
	
	\item The interior of $J$ includes $\bigcap_{G \in \mathcal{M}} G^+ \cap E^-$.\label{cl2}
	
	\item For each $G \in \mathcal{M}$, $G^+$ includes $C$.\label{cl3}
	
	\item The interior of $J$ contains no point of $\sigma$.\label{cl4}
\end{enumerate}
Claims (\ref{cl2}) and (\ref{cl3}) together imply that the interior of $J$ includes $C$.  Claim (\ref{cl1}) will be used to prove (\ref{cl4}).  Claim (\ref{cl4}) shows that the interior of $J$ is included in a connected component of $D - \sigma$ which then must be $C$.

We begin by proving (\ref{cl1}).  Let $p' \in G^-$.  Let $z_0 \in \C - \overline{D}$.  Thus, $z_0$ is exterior to $J$ since $J \subseteq \overline{D}$.  We construct an arc from $p'$ to $z_0$ that contains no point of $J$.  Let $q \in \lambda_G - \overline{G}$.  By Lemma \ref{lm:SIDES}, $G^-$ is the interior of $G \cup \lambda_G$.  So, there is an arc $\sigma_1$ from $p'$ to $q$ so that $\sigma' \cap \partial G^- = \{q\}$.  There is an arc $\sigma_2$ from $q$ to $z_0$ so that $\sigma_2 \cap \partial D = \{q\}$.  Thus, $\sigma_1 \cup \sigma_2$ is an arc from $p'$ to $z_0$ that contains no point of $J$.  Thus, $p'$ is exterior to $J$ for every $p' \in G^-$.

We now prove (\ref{cl2}).  Suppose $p_0 \in E^-$ belongs to $G^+$ for every $G \in \mathcal{M}$.  By way of contradiction, suppose $p_0$ is exterior to $J$.  Again, let $z_0 \in \C - \overline{D}$.  Thus, the exterior of $J$ includes an arc from $p_0$ to $z_0$; let $\alpha$ denote such an arc.  
By examination of cases, $\alpha$ cannot cross the boundary of $D$ at any boundary point of $E^-$.  So, it must do so at a boundary point of $E^+$.  But, this entails that $\alpha$ crosses $E$ which it does not since $J$ includes $E$.  This is a contradiction, and so $p_0$ is interior to $J$.  

Next, we prove (\ref{cl3}).  Let $G \in \mathcal{M}$.  Since $\sigma$ is locally connected, and since $p \in E$, there is a positive number $\delta$ so that $D_\delta(p)$ contains no point of any connected component of $\sigma \cap E^-$.  However, this disk must contain a point of $C$, $p'$.  So, $[p',p]$ contains a point of $E$ but no point of $G$.  Hence, $p' \in G^+$.  Since $C$ is a connected subset of $D - G$, $C \subseteq G^+$.

Finally, we prove (\ref{cl4}).  By way of contradiction, suppose $p'$ is a point on $\sigma$ that is interior to $J$.  As noted above, $E^-$ includes the interior of $J$.  So, $p' \in \sigma \cap E^-$.  Let $G_1$ be the connected component of $p'$ in $\sigma \cap E^-$.  Let $G$ be a $\preceq$-maximal component of $\sigma \cap E^-$ so that $G_1 \preceq G$.  Since $p'$ lies inside $J$, and since $J$ includes $G$, $p' \not \in G$.  So, $G_1 \prec G$.  This means that $G_1 \subseteq G^-$.  By (\ref{cl1}), $p'$ is exterior to $J$- a contradiction.  So, the interior of $J$ contains no point of $\sigma$.  

By the remarks after (\ref{cl4}), $C$ is the interior of $J$ and the proof is complete.
\end{proof}

\begin{proof}[Proof of Theorem \ref{thm:MLC}]
Let $C$ be the connected component of $z_0$ in $D - \sigma$.  Let $l = [z_0, \zeta_0]$.  Let $z_1$ be the point in $l \cap \sigma$ that is closest to $z_0$.  Thus, $z_1 \in \partial C$.  Since $|z_1 - \zeta_0| < 2^{-g(k)}$, $\sigma$ includes an arc from $z_1$ to $\zeta_0$ whose diameter is smaller than $2^{-k}$; call this arc $\sigma_1$.  

We claim that $D$ includes $\sigma_1$.  For, let $q \in \sigma_1$.  It follows that 
\[
\max\{|q - z_0|, |q - \zeta_0| \} < 2^{-k} + 2^{-g(k)}.
\] 
Since $2^{-k} + 2^{-g(k)} \leq \max\{d(\zeta_0, \partial D), d(z_0, \partial D)\}$, it follows that $q \in D$.

Since $\sigma_1 \subseteq D$, $\zeta_0$ belongs to the connected component of $z_1$ in $D \cap \sigma$.  By the `Furthermore' part of Theorem \ref{thm:COMP.JORDAN}, the boundary of $C$ includes this component.  Thus, $\zeta_0$ is a boundary point of $C$.
\end{proof}

\begin{proof}[Proof of Theorem \ref{thm:BOUNDARY.COMP}]
Without loss of generality, suppose $D_r(\zeta_0)$ does not include $D$.  Let $J$ denote the boundary of $D$.  It follows that $\partial D_r(\zeta_0)$ separates two points of $J$.  

It follows from Theorem \ref{thm:MLC} that $\zeta_0$ is a boundary point of at least one connected component of $D_r(\zeta_0) - J$.  We now show it is a boundary point of exactly two such components.  Let $E$ be the connected component of $\zeta_0$ in $D_r(\zeta_0) \cap J$.  Thus, as noted in the proof of Theorem \ref{thm:COMP.JORDAN}, $\overline{E}$ is a crosscut of $D_r(\zeta_0)$.  If $C$ is a connected component of $D_r(\zeta_0) - J$, and if $\zeta_0$ is a boundary point of $C$, then exactly one side of $E$ includes $C$.  By the proof of Theorem \ref{thm:BOUNDARY.COMP}, if $C$ is a connected component of $D_r(\zeta_0) - J$, then the side of $E$ that includes $C$ completely determines the boundary of $C$.  Thus, $\zeta_0$ is a boundary point of exactly two connected components of $D - J$; one for each side of $E$.

So, let $C_1$, $C_2$ denote the two connected components of $D_r(\zeta_0) - J$ whose boundaries contain $\zeta_0$.  Each of these components is a connected subset of $\C - J$.  So each is either included in the interior of $J$ or in the exterior of $J$.  Since there are points of the interior and exterior of $J$ that are arbitrarily close to $\zeta_0$, it follows from Theorem \ref{thm:MLC} that one of these components is included in the interior of $J$ and one is included in the exterior of $J$.  Suppose $C_1$ is included in the interior of $J$; that is, $D \supseteq C_1$.

Let $p \in C_1$, and let $U$ be the connected component of $p$ in $D \cap D_r(\zeta_0)$.  We show that $U = C_1$.  Since $C_1$ is a connected subset of $D \cap D_r(\zeta_0)$ that contains $p$, $C_1 \subseteq U$.  Since $U$ is a connected subset of $D_r(\zeta_0) - J$ that contains $p$, $U \subseteq C_1$.  This completes the proof of the theorem.
\end{proof}

\section{Preliminaries to proof of Theorem \ref{thm:MODULUS.ESTIMATE}: polar separations}\label{sec:POLAR}

\begin{definition}\label{def:POLAR.SEP}
Let $\mathcal{A}$ be an annulus, and let $\Omega$ be an open subset of $\mathcal{A}$.  A \emph{polar separation} of the boundary of $\Omega$ is a pair of disjoint sets $(E,F)$ so that 
 whenever $C$ is an intermediate circle of $\mathcal{A}$, there is a connected component of $C \cap \Omega$ whose boundary contains a point of $E$ and a point of $F$.
\end{definition}

Our goal in this section is to prove the following.

\begin{theorem}\label{thm:POLAR}
Let $\mathcal{A}$ be an annulus, and let $D$ be a simply connected Jordan domain.  Suppose that 
$\mathcal{A}$ separates two boundary points of $D$, and let $\gamma_1$ and $\gamma_2$ be the subarcs of the boundary of $D$ that join these points. 
Then, $(\gamma_1 \cap \mathcal{A}, \gamma_2 \cap \mathcal{A})$ is a polar separation of the boundary of $D \cap \mathcal{A}$.
\end{theorem}

Our proof of Theorem \ref{thm:POLAR} is based on the following lemma.

\begin{lemma}\label{lm:CUTS}
Let $C$ be a circle, and let $D$ be a simply connected Jordan domain.  Suppose
$C$ separates two boundary points of $D$.  
Then, there is a connected component of $C \cap D$ 
whose boundary hits both subarcs of the boundary of $D$ that join these two boundary points of $D$.
\end{lemma}

\begin{proof}
Let $p$ be a boundary point of $D$ that is exterior to $C$, and let $q$ be a boundary point of $D$ that belongs to the interior of $C$.  

Let $\gamma_1$, $\gamma_2$ denote the subarcs of the boundary of $D$ that join $p$ and $q$.  Let $\alpha$ be a crosscut of $D$ so that $\alpha \cap C$ consists of a single point; label this point $p'$.
Let $D_j$ denote the interior of $\alpha \cup \gamma_j$.  By Lemma \ref{lm:SIDES}, $D_1$ and $D_2$ are the sides of $\alpha$.

Now, for each $j \in \{1,2\}$, we construct a point $q_j$ in $C \cap D_j$ so that $p'$ is a boundary point of the connected component of $q_j$ in $C \cap D_j$.  Since $D$ is open, there is a positive number $\delta$ so that $D_\delta(p') \subseteq D$.  Let $C' = C \cap D_\delta(p')$.  Thus, $C'$ is a subarc of $C$.
Let $q \in C' - \{p'\}$.  Then, $q \not \in \alpha$ since $C \cap \alpha = \{p'\}$.  
So, $q \in D_1 \cup D_2$.  Without loss of generality, suppose $q \in D_1$.   Relabel $q$ as $q_1$.  Let $q_2$ be a point of $C'$ so that $p'$ is between $q_1$ and $q_2$ on $C'$.  Again, $q_2 \in D_1 \cup D_2$.  Since $D_1$ is the interior of a Jordan curve, and since the subarc of $C'$ from $q_1$ to $q_2$ crosses the boundary of $D_1$ exactly once, $q_2 \not \in D_1$.  
So, $q_2 \in D_2$.  

Let $E_j$ denote the connected component of $q_j$ in $C \cap D_j$.  By construction, $p'$ is a boundary point of $E_j$.  So, the other endpoint of $E_j$ must be in $\gamma_j$ since $C \cap \alpha = \{p'\}$.  Set $E = E_1 \cup E_2$.  Thus, $E$ is a connected component of $C \cap D$.  One endpoint of $E$ belongs to $\gamma_1$, and the other belongs to $\gamma_2$.  This proves the lemma.  
\end{proof}

\begin{proof}[Proof of Theorem \ref{thm:POLAR}]
By assumption, $\mathcal{A}$ separates two boundary points of $D$.  One of these points is interior to the inner circle of $\mathcal{A}$, and the other is exterior to the outer circle of $\mathcal{A}$.  Let $p$ denote a point that is exterior to the outer circle of $\mathcal{A}$, and let $q$ denote a point that is interior to the inner circle of $\mathcal{A}$.  

Let $C$ be an intermediate circle of $\mathcal{A}$.  Then, $p$ is exterior to $C$ and $q$ is interior to $C$.  So, by Lemma \ref{lm:CUTS}, there is a connected component of $C \cap D$ so that one of its endpoints lies on $\gamma_1$ and the other lies on $\gamma_2$.  Thus, $(\gamma_1 \cap \mathcal{A}, \gamma_2 \cap \mathcal{A})$ is a polar separation of the boundary of $D \cap \mathcal{A}$.
\end{proof}

\section{Proof of Theorems \ref{thm:MODULUS.ESTIMATE} and \ref{thm:DIAM.EST}}\label{sec:ESTIMATES}

When $X, Y \subseteq \C$, let $\dinf(X,Y)$ denote the infimum of $|z - w|$ as $z$ ranges over all points of $X$ and $w$ ranges over all points of $Y$.

The proof of the following is essentially the same as the proof of Lemma 4.1 of \cite{McNicholl.2014} which is a standard length-area argument.

\begin{lemma}\label{lm:POLAR.SEP}
Let $\mathcal{A}$ be an annulus, and let $\Omega$ be an open subset of $\mathcal{A}$.  Suppose $(E,F)$ is a polar separation of the boundary of $\Omega$.  Then,
\[
\lambda(\mathcal{A}) \geq \sup_\phi \frac{\dinf(\phi[E], \phi[F])^2}{\Area(\phi[\Omega])}
\]
where $\phi$ ranges over all maps that are conformal on a neighborhood of $\overline{\Omega}$.
\end{lemma}

\begin{proof}[Proof of Theorem \ref{thm:MODULUS.ESTIMATE}]
Note that $r < 1$ since $C$ is non-empty.  

We begin by constructing a rectangle $R$ as follows.   Let $z_0$ be any point of $\phi[C]$.  Choose $m, l_0$ so that $l_0 > l$, $m > \sqrt{\pi \lambda(\mathcal{A})}$, and $(1 - l_0)^2 + m^2 < ( 1 - l)^2 + \pi \lambda(\mathcal{A})$.   Since $r^2 = ( 1 - l)^2 + \pi \lambda(\mathcal{A})$, $z$ is exterior to the outer circle of $\mathcal{A}$ whenever $|\phi(z)| \leq \sqrt{(1 - l_0)^2 + m^2}$.  Let:
\begin{eqnarray*}
\nu_1 & = & \frac{z_0}{|z_0|} (1 - l_0 + m i)\\
\nu_2 & = & \frac{z_0}{|z_0|} (1 - l_0 - m i)\\
\end{eqnarray*}
Thus, the radius $[0, z_0/|z_0|]$ is a perpendicular bisector of the line segment $[\nu_1, \nu_2]$.  The midpoint of $[\nu_1, \nu_2]$ is $(1 - l_0) z_0/ |z_0|$, and the length of $[\nu_1, \nu_2]$ is $2m$.   Let:
\begin{eqnarray*}
\nu_3 & = & \frac{z_0}{|z_0|} (1 + m i )\\
\nu_4 & = & \frac{z_0}{|z_0|} (1 - m i )
\end{eqnarray*}
Thus, the line segment $[\nu_3, \nu_4]$ is perpendicular to the radius $[0, z_0/|z_0|]$.  Furthermore, the length of this segment is $2m$ and its midpoint is $z_0/ |z_0|$.

Let $R$ be the open rectangle whose vertices are $\nu_1$, $\nu_2$, $\nu_3$, and $\nu_4$.  That is, $R$ is the interior of $[\nu_1, \nu_3] \cup [\nu_3, \nu_4] \cup [\nu_4, \nu_2] \cup [\nu_2, \nu_1]$.  

Note that the diameter of $R$ is $\sqrt{l_0^2 + 4 m^2}$.  
Also, the diameter of $R$ approaches $\sqrt{l^2 + 4 \pi \lambda(\mathcal{A}) }$ as $(l_0,m) \rightarrow (l, \sqrt{\pi \lambda(\mathcal{A})})$.  It thus suffices to show that $\phi[C] \subseteq R$.  

We claim that it suffices to show that $\phi[C]$ contains no boundary point of $R$.  
For, since $\phi^{-1}(z_0)$ is interior to the outer circle of $\mathcal{A}$, the modulus of $z_0$ is larger than $\sqrt{(1 - l_0)^2 + m^2}$ which is larger than $1 - l_0$.  This implies that $z_0 \in R$.  
Since $R$ contains at least one point of $\phi[C]$, namely $z_0$, and since $\phi[C]$ is connected, it suffices to show that $\phi[C]$ contains no boundary point of $R$.  

Since $[\nu_3, \nu_4]$ contains no point of the unit disk, it contains no point of $\phi[C]$.  By construction, $|\nu_1|= |\nu_2| = \sqrt{(1 - l_0)^2 + m^2}$.  Thus, $|z| \leq \sqrt{(1 - l_0)^2 + m^2}$ whenever $z \in [\nu_1, \nu_2]$.  It follows from what has been observed about $l_0$ and $m$ that $[\nu_1, \nu_2]$ contains no point of $\phi[C]$.  So, it suffices to show that $[\nu_1, \nu_3] \cup [\nu_4, \nu_2]$ contains no point of $\phi[C]$.  

Let us begin by showing that $[\nu_1, \nu_3]$ contains no point of $\phi[C]$.  By way of contradiction, suppose otherwise.  
In order to obtain a contradiction, we construct a Jordan curve $J$ 
so that $\mathcal{A}$ separates two points of $J$ as follows.  Let $z_1$ be a point of $\phi[C]$ that belongs to $[\nu_1, \nu_3]$.  Thus, by what has just been observed, $z_1 \neq \nu_1$.  Let $\sigma_0$ be the pre-image of $\phi$ on $[\nu_1, 0]$.  Let $\sigma_1'$ be the pre-image of $\phi$ on $[\nu_1, z_1]$.  Let $\sigma_3'$ be the pre-image of $\phi$ on $[0, z_0]$.  Since $C$ is connected, it includes an arc from $\phi^{-1}(z_1)$ to $\phi^{-1}(z_0)$; label this arc $\sigma_2'$.  Let $w_1$ be the first point on $\sigma_1'$ that belongs to $\sigma_2'$.  Let $w_2$ be the first point on $\sigma_3'$ that belongs to $\sigma_2'$.  Let $\sigma_1$ be the subarc of $\sigma_1'$ from $\phi^{-1}(\nu_1)$ to $w_1$, and let $\sigma_3$ be the subarc of $\sigma_3'$ from $w_2$ to $\phi^{-1}(0)$.  Let $\sigma_2$ be the subarc of $\sigma_2'$ from $w_1$ to $w_2$.  Let $J = \sigma_0 \cup \sigma_1 \cup \sigma_2 \cup \sigma_3$.  Thus, $J$ is a Jordan curve.  By construction, $\mathcal{A}$ separates two points of $J$.

Let $D'$ denote the interior of $J$.  Let $\Omega = D' \cap \mathcal{A}$.  
Let $E = \sigma_1 \cap \mathcal{A}$, and let $F = \sigma_3 \cap \mathcal{A}$.  We claim that $(E,F)$ is a polar separation of the boundary of $\Omega$.  
For, let $p = \phi^{-1}(\nu_1)$, and let $q = w_1$ (where $w_1$ is as in the construction of $J$).  Thus, $p$ is exterior to the outer circle of $\mathcal{A}$.  Since $q \in C$, $q$ is interior to the inner circle of $\mathcal{A}$.  Let $\gamma_1 = \sigma_1$, and let $\gamma_2 = \sigma_2 \cup \sigma_3 \cup \sigma_0$.  Therefore, $\gamma_1$, $\gamma_2$ are the subarcs of the boundary of $D'$ that join $p$ and $q$.  So, by 
Theorem \ref{thm:POLAR}, 
$(\gamma_1 \cap \mathcal{A}, \gamma_2 \cap \mathcal{A})$ is a polar separation of the boundary of $\Omega$.  
Since $\sigma_0$ is the pre-image of $\phi$ on $[\nu_1, 0]$, $\sigma_0$ contains no point of $\overline{\mathcal{A}}$.  
Since $\sigma_2 \subseteq C$, $\sigma_2$ contains no point of $\overline{\mathcal{A}}$.  
Thus, $E = \gamma_1 \cap \mathcal{A}$, and $F = \gamma_2 \cap \mathcal{A}$.  Hence, $(E,F)$ is a polar separation of the boundary of $\Omega$.   

By construction, $\dinf(\phi[E], \phi[F]) = m$.  So, by Lemma \ref{lm:POLAR.SEP}, the area of $\phi[\Omega]$ is at least as large as 
\[
m^2 \lambda(\mathcal{A})^{-1} > \pi.
\]
This is impossible since the unit disk includes $\phi[\Omega]$.  Thus, $[\nu_1, \nu_3]$ contains no point of $\phi[C]$.  

By similar reasoning, $[\nu_4, \nu_2]$ contains no point of $\phi[C]$.  Thus, $\phi[C] \subseteq R$, and the theorem is proven.
\end{proof}

\begin{proof}[Proof of Theorem \ref{thm:DIAM.EST}]
Suppose $r_0$ is a positive number that is smaller than (\ref{eqn:DIAM.EST}).  We begin by defining an annulus $\mathcal{A}$ as follows.  Choose $l$ so that $0 < l < \epsilon$ and so that 
\[
r_0 < \exp\left( \frac{8\pi^2}{l^2 - \epsilon^2} \right) \min\left\{|\zeta_0 - \phi^{-1}(w)|\ :\ |w| \leq \sqrt{(1 - l)^2 + \frac{ \epsilon^2 - l^2}{4}}\right\}.
\]
There is a positive number $r_1$ so that 
\[
r_1 < \min\left\{ |\zeta_0 - \phi^{-1}(w)|\ :\ |w| \leq \sqrt{(1 - l)^2 + \frac{1}{4}(\epsilon^2 - l^2)}\right\}
\]
and so that 
\[
r_0 < \exp\left( \frac{8 \pi^2}{l^2 - \epsilon^2} \right) r_1.
\]
Since $l < \epsilon$, $r_0 < r_1$.  So, define $\mathcal{A}$ to be the annulus whose center is $\zeta_0$, whose outer radius is $r_1$, and whose inner radius is $r_0$.  

We now show that the diameter of $\phi[C(D; \zeta_0, r_0)]$ is smaller than $\epsilon$.  First, note that $\pi \lambda(\mathcal{A}) < (\epsilon^2 - l^2)/4$.  Set
$r = \sqrt{(l -1)^2 + \pi \lambda(\mathcal{A})}$.  It follows that $|\zeta_0 - z| > r_1$ whenever $|\phi(z)| \leq r$.  For, if $|\phi(z)| \leq r$, then $|\phi(z)| < \sqrt{(l - 1)^2 + (\epsilon^2 - l^2) / 4}$ and so $r_1 < |\zeta_0 - z|$ by the choice of $r_1$.  This means that $\mathcal{A}$ separates its center from $\phi^{-1}[\overline{D_r(0)}]$.  
By Theorem \ref{thm:MODULUS.ESTIMATE}, the diameter of $\phi[C(D; \zeta_0, r_0)]$ is at most 
\[
\sqrt{l^2 + 4\pi\lambda(\mathcal{A}) }.
\]
We have
\begin{eqnarray*}
l^2 + 4\pi\lambda(\mathcal{A}) & = & l^2 + \frac{8 \pi^2}{\log(r_1 / r_0)}\\
& < & l^2 + \epsilon^2 - l^2 = \epsilon^2.
\end{eqnarray*}
Thus, the diameter of $\phi[C(D; \zeta_0, r_0)]$ is smaller than $\epsilon$.
\end{proof}

\section{Proof of the Carath\'eodory Theorem}\label{sec:CT}

We now conclude with the proof of Theorem \ref{thm:CT}.  Set $r_0 = 2^{-k} + 2^{-g(k)}$.  By Theorem \ref{thm:MLC}, $z_0 \in C(D; \zeta_0, r_0)$.  By Theorem \ref{thm:DIAM.EST}, $|\phi(z_0) - \phi(\zeta_0)| < \epsilon$.  Thus, $\lim_{z \rightarrow \zeta_0} \phi(z) = \phi(\zeta_0)$.

We now show that this extension of $\phi$ is injective.  It suffices to show that $\phi(\zeta_0) \neq \phi(\zeta_1)$ whenever $\zeta_0$ and $\zeta_1$ are distinct boundary points of $D$.  By way of contradiction, suppose $\phi(\zeta_0) = \phi(\zeta_1)$.  Let $p = \phi(\zeta_0)$.  

We construct a Jordan curve $\sigma$ as follows.  Let $\alpha$ be a crosscut of $D$ that joins $\zeta_0$ and $\zeta_1$.  Thus, $\phi[\alpha]$ is a Jordan curve that contains no unimodular point other than $p$.  Let $\sigma = \phi[\alpha]$.  

We now construct an annulus $\mathcal{A}$ that separates two points of $\sigma$.  Fix a positive number $R$ so that 
$R < \max\{|z - p|\ :\ z \in \sigma\}$.  Choose another positive number $r$ so that $r < R$.  Let $\mathcal{A}$ be the annulus whose center is $p$, whose inner radius is $r$, and whose outer radius is $R$.  By the choice of $R$, there is a point $q \in \sigma$ that is exterior to the outer circle of $\mathcal{A}$.  Let $\gamma_1$ and $\gamma_2$ be the subarcs of $\sigma$ that join $p$ and $q$.  Let $E = \gamma_1 \cap \mathcal{A}$, and let $F = \gamma_2 \cap \mathcal{A}$.  
Finally, let $\Omega = \mathcal{A} \cap \D$ (where $\D$ is the unit disk).  Then, by Theorem \ref{thm:POLAR}, $(E,F)$ is a polar separation of the boundary of $\Omega$.  
Now, since $R$ is fixed, as $r \rightarrow 0^+$, $\lambda(\mathcal{A}) \rightarrow 0$.  However, by the choice of $R$, $\dinf(E,F)$ is bounded away from $0$ as $r \rightarrow 0^+$. Thus, by Lemma \ref{lm:POLAR.SEP} (applied to $\phi^{-1}$), $\Area(\phi^{-1}[\Omega]) \rightarrow \infty$ as $r \rightarrow 0^+$.  Since $\phi^{-1}[\Omega] \subseteq D$, this is a contradiction.  Thus, $\phi(\zeta_0) \neq \phi(\zeta_1)$.

Finally, we show that this extension of $\phi$ is surjective.  Let $\zeta$ be a point on the unit circle.  It follows from the Balzano-Weierstrauss Theorem that there is a boundary point of $D$, $\zeta_1$, so that $\zeta_1 \in \overline{\{ \phi^{-1}(r \zeta)\ :\ 0 < r < 1\}}$.  Thus, $\phi(\zeta_1) = \zeta$ by the continuity of $\phi$.

\section*{Acknowledgement}

I thank the referee for useful comments and Valentin Andreev for helpful conversation.

\def\cprime{$'$}
\providecommand{\bysame}{\leavevmode\hbox to3em{\hrulefill}\thinspace}
\providecommand{\MR}{\relax\ifhmode\unskip\space\fi MR }
\providecommand{\MRhref}[2]{%
  \href{http://www.ams.org/mathscinet-getitem?mr=#1}{#2}
}
\providecommand{\href}[2]{#2}

\end{document}